\documentclass[11pt, a4paper]{article}
\usepackage{}
\usepackage{amsthm}
\usepackage{mathrsfs}
\usepackage{amsmath,amssymb,latexsym,color}
\usepackage[colorlinks,
linkcolor=blue,
anchorcolor=green,
citecolor=magenta
]{hyperref}
\usepackage{graphicx}
\usepackage{tikz}
\usetikzlibrary{calc}
\usepackage{CJK}

\usepackage{authblk}

\long\def\delete#1{}

\usepackage{color}

\definecolor{Blue}{rgb}{0,0,1}
\definecolor{Red}{rgb}{1,0,0}
\definecolor{DarkGreen}{rgb}{0,0.6,0}
\definecolor{DarkYellow}{rgb}{1,1,0.2}
\definecolor{DarkPurple}{rgb}{.6,0,1}

\usepackage{xcolor}
\usepackage[normalem]{ulem}

\usepackage{cleveref}
\crefformat{section}{\S#2#1#3}
\crefformat{subsection}{\S#2#1#3}
\crefformat{subsubsection}{\S#2#1#3}
\crefrangeformat{section}{\S\S#3#1#4 to~#5#2#6}
\crefmultiformat{section}{\S\S#2#1#3}{ and~#2#1#3}{, #2#1#3}{ and~#2#1#3}

\def\ma{\mathscr{A}}
\def\mb{\mathscr{B}}

\def\mf{\mathscr{F}}

\def\mh{\mathscr{H}}

\def\mt{\mathscr{T}}

\def\bs{\setminus}

\def\ff{\mathbb{F}_q}

\def\ge{\geqslant}
\def\le{\leqslant}
\def\b{\brack}

\def\ro{\romannumeral}

\numberwithin{equation}{section}

\newtheorem{thm}{Theorem}[section]
\newtheorem{lem}[thm]{Lemma}

\allowdisplaybreaks

\begin{document}
	\setcounter{page}{1}
	\renewcommand{\thefootnote}{}
	\newcommand{\remark}{\vspace{2ex}\noindent{\bf Remark.\quad}}
\renewcommand{\abovewithdelims}[2]{%
		\genfrac{[}{]}{0pt}{}{#1}{#2}}

	%-------------------  First Head  -----------------------------------------
	
	\def\qed{\hfill$\Box$\vspace{11pt}}
	
	\title {\bf  More on $r$-cross $t$-intersecting families for vector spaces}

	\author[a]{Tian Yao\thanks{E-mail: \texttt{tyao@hist.edu.cn}}}
	%\author{Benjian Lv\thanks{E-mail: \texttt{bjlv@bnu.edu.cn}}}
	\author[b]{Dehai Liu\thanks{E-mail: \texttt{liudehai@mail.bnu.edu.cn}}}
	\author[b]{Kaishun Wang\thanks{E-mail: \texttt{wangks@bnu.edu.cn}}}
	\affil[a]{School of Mathematical Sciences, Henan Institute of Science and Technology, Xinxiang 453003, China}
	\affil[b]{Laboratory of Mathematics and Complex Systems (Ministry of Education), School of
		Mathematical Sciences, Beijing Normal University, Beijing 100875, China}

	\date{}
	
	\openup 0.5\jot
	\maketitle

	\begin{abstract}
		Let $V$ be a finite dimensional vector space over a finite field. 
		Suppose that $\mf_1$, $\mf_2$, \dots, $\mf_r$ are $r$-cross $t$-intersecting families of $k$-subspaces of $V$.  
		In this paper, we determine the extremal structure when $\prod_{i=1}^r|\mf_i|$ is maximum under the condition that  $\dim(\bigcap_{F\in\mf_i}F)<t$ for each $i$. 
		%Our main result can be seen as a product version of the Hilton-Milner theorem for vector spaces.
		
		\vspace{2mm}
		
		
		\
		
		
		%\footnotetext{ E-mail address: caomengyu@mail.bnu.edu.cN'(M.Cao), bjlv@bnu.edu.cN'(B.Lv), wangks@bnu.edu.cN'(K.Wang)}
		
	\end{abstract}
	
	\section{Introduction}

		Intersection problems originate from the famous Erd\H{o}s-Ko-Rado theorem \cite{EKR}. 
		For decades, these problems have attracted widespread attention. We refer readers to \cite{SHMT,SEKRT,CROOSBORG2016,cao-rcrossset,FW1,FW2,SHM,tokushige-set-cross-product,SCROSSSUM2013,Wn,kong} for rich results of sets. 
		Intersection problems are also studied on some other mathematical objects, for example, vector spaces.

	Let  $V$  be an $n$-dimensional vector space over the finite field $\ff$.
	  Denote the family of all $k$-subspaces of $V$ by ${V\b k}$. 
A family $\mf\subseteq{V\b k}$ is called \emph{$t$-intersecting} if $\dim(F_1\cap F_2)\ge t$ for any $F_1,F_2\in\mf$. 
The Erd\H{o}s-Ko-Rado Theorem for vector spaces \cite{VEKR2,VEKR1,VEKR3} shows that each member of a maximum-sized $t$-intersecting subfamily of ${V\b k}$ contains a fixed $t$-subspace of $V$ when $n>2k$. 
In \cite{VHM1,VHM2,VHMNEW,VHMNEW2}, the Hilton-Milner theorem for vector spaces was proved, providing the structure of the second largest maximal $t$-intersecting subfamilies of ${V\b k}$.

Let $r\ge2$. 
Suppose that $\mf_1,\mf_2,\dots,\mf_r$ are subfamilies of ${V\b k}$ with $\dim(\bigcap_{i=1}^rF_i)\ge t$ for any  $F_i\in\mf_i$, $1\le i\le r$. 
We say they are \emph{$r$-cross $t$-intersecting families}. 
Observe that $r$-cross $t$-intersecting families  generalize $t$-intersecting families. 
Indeed, if $r=2$ and $\mf_1=\mf_2=\mf$, then $\mf$ is $t$-intersecting. 
A natural problem is maximizing $\sum_{i=1}^r|\mf_i|$. 
Wang and Zhang \cite{SCROSSSUM2013} completely determined the extremal structure when $r=2$. 
Another natural problem focuses on $\prod_{i=1}^r|\mf_i|$. 
An Erd\H{o}s-Ko-Rado type theorem \cite{CAO} states that, for large $n$,  $\prod_{i=1}^r|\mf_i|$ is maximum implies that $\dim(\bigcap_{i=1}^r\bigcap_{F\in\mf_i}F)=t$.  
	In \cite{CAO}, assuming $r=2$ and $\dim(\bigcap_{i=1}^2\bigcap_{F\in\mf_i}F)<t$, the authors also described $\mf_1$ and $\mf_2$ when $|\mf_1||\mf_2|$ takes the maximum value.

		For $T\in{V\b t}$, $L,M\in{V\b k+1}$ and $Z\in{V\b t+2}$ with $T\subseteq L\cap M$ and $\dim(L\cap M)\ge t+2$, write
	\begin{equation*}
		\begin{aligned}
			\mh_1(T,L,M)&=\left\{F\in{V\b k}: T\subseteq F,\ \dim(F\cap L)\ge t+1\right\}\cup\left\{F\in{M\b k}: T\not\subseteq F\right\},\\
			\mh_2(Z)&=\left\{F\in{V\b k}:\dim(F\cap Z)\ge t+1\right\}.
		\end{aligned}
	\end{equation*}
	Observe that if $\mf_1=\mh_1(T,L,M)$ and $\mf_2=\mh_1(T,M,L)$, or $\mf_1=\mf_2=\mh_2(Z)$, then $\mf_1$ and $\mf_2$ are $2$-cross $t$-initersecting families with $\dim(\bigcap_{F\in\mf_i}F)<t$, $i=1,2$.

	In this paper,  we characterize the structure of $\mf_1$, $\mf_2$, \dots, $\mf_r$ when $\prod_{i=1}^r|\mf_i|$ is maximum under the condition that $\dim(\bigcap_{F\in\mf_i}F)<t$ for each $i$.

\begin{thm}\label{main1}
	Let $n$, $k$ and $t$ be positive integers with $k\ge t+1$ and $n\ge2k+2t+5$. 
	Assume that  $\mf_1,\mf_2\subseteq{V\b k}$ are $2$-cross $t$-intersecting families with $\dim(\bigcap_{F\in\mf_i}F)<t$ for each $i$. If $|\mf_1||\mf_2|$ takes the maximum value, then one of the following holds.
	\begin{itemize}
		\item[\rm{(\ro1)}] $\mf_1=\mh_1(T,L,M)$ and $\mf_2=\mh_1(T,M,L)$ for some $T\in{V\b t}$ and $L,M\in{V\b k+1}$ with $T\subseteq L\cap M$ and $\dim(L\cap M)\ge t+2$.
		\item[\rm{(\ro2)}] $\mf_1=\mf_2=\mh_2(Z)$ for some $Z\in{V\b t+2}$.
	\end{itemize}
	Morover, if $k>2t+1$, then $(\ro1)$ holds; if $k\le2t+1$, then $(\ro2)$ holds, or $(k,t)=(3,1)$ and $(\ro1)$ holds.
\end{thm}

%The following theorem is for the case $r\ge3$, from which we get the structure of maximum-sized \emph{non-trivial $r$-wise $t$-intersecting} families, see \cite{CAO} for more details.

\begin{thm}\label{main2}
	Let $n$, $r$, $k$ and $t$ be positive integers with $k-t+1\ge r\ge3$ and $n\ge2k+2t+2r+1$. 
	Assume that $\mf_1,\mf_2,\dots,\mf_r\subseteq{V\b k}$ are $r$-cross $t$-intersecting families with $\dim(\bigcap_{F\in\mf_i}F)<t$ for each $i$, and $\prod_{i=1}^r|\mf_i|$ takes the maximum value.
	\begin{itemize}
		\item[\rm{(\ro1)}] If $k>2t+2r-3$, then $\mf_i=\mh_1(T,M,M)$ $(i=1,2,\dots,r)$ for some $T\in{V\b t+r-2}$ and $M\in{V\b k+1}$ with $T\subseteq M$.
		\item[\rm{(\ro2)}] If $k\le2t+2r-3$, then $\mf_i=\mh_2(Z)$ $(i=1,2,\dots,r)$ for some $Z\in{V\b t+r}$.
	\end{itemize}
\end{thm}

	The first theorem can be seen as a product version of the Hilton-Milner theorem for vector spaces. 
The second implies the structure of maximum-sized {non-trivial $r$-wise $t$-intersecting families} in \cite{CAO}.

\section{Inequalities concerning $2$-cross $t$-intersecting families}

In this section, we show some inequalities concerning $2$-cross $t$-intersecting families in preparation for the proof of our main results. 
The \emph{Gaussian binomial coefficient} is defined by
$${n\b k}:=\prod_{0\le i<k}\dfrac{q^{n-i}-1}{q^{k-i}-1}.$$
We begin with a lemma about it, which can be easily checked.

\begin{lem}\label{gaosifangsuo}
	Let $m$ and $i$ be positive integers with $i<m$. Then 
 $$q^{m-i}<\frac{q^m-1}{q^i-1}<q^{m-i+1},\quad q^{i(m-i)}<{m\b i}<q^{i(m-i+1)}.$$
\end{lem}

Let $T$ be a subspace of $V$ and $\mf\subseteq{V\b k}$. 
From now on, write
$\mf_T=\{F\in\mf: T\subseteq F\}$, and ``cross $t$-intersecting" instead of ``$2$-cross $t$-intersecting". 
If $k\ge t$ and $\dim(T\cap F)\ge t$ for each $F\in\mf$, then $T$ is called a \emph{$t$-cover} of $\mf$.  
Define the \emph{$t$-covering number} $\tau_t(\mf)$ of $\mf$ as the minimum dimension of a $t$-cover of $\mf$. 
Observe that $\dim(\bigcap_{F\in\mf}F)<t$  if and only if $t+1\le\tau_t(\mf)\le n$.  

\begin{lem}\label{t+1}
	Let $n$, $k$ and $t$ be positive integers with $k\ge t+1$ and  $n\ge2k+2t+5$. 
	Suppose that $\mf_1,\mf_2\subseteq{V\b k}$ are cross $t$-intersecting families with $\tau_t(\mf_1)=\tau_t(\mf_2)=t+1$, and $\mb$ is the set of  all members of $\mf_2$ not containing any $(t+1)$-dimensional $t$-cover of $\mf_1$. Then $$|\mb|\le q^{k-3t-1}{n-t-1\b k-t-1}.$$
\end{lem}
\begin{proof}
		We may assume that $\mb\neq\emptyset$, and  $S\in{V\b t+1}$ is a $t$-cover of $\mf_2$. 
Observe that 
\begin{equation}\label{aa1}
	\mb\subseteq\bigcup_{U\in{S\b t}}\mb_U.
\end{equation}
Let $U\in{S\b t}$ with $\mb_U\neq\emptyset$. 
Since $\tau_t(\mf_1)=t+1$, there exists $F_1\in\mf_1$ with $\dim(U\cap F_1)\le t-1$. 
Write $\alpha=\dim(U\cap F_1)$ and $H=U+F_1$. 
For each $G\in\mb_U$, by $\dim(G\cap F_1)\ge t$, we have $\dim(H\cap G)\ge2t-\alpha\ge t+1$. 
Hence
\begin{equation}\label{aa2}
	\mb_U\subseteq\bigcup_{R\in{H\b t+1},\ U\subseteq R}\mb_R.
\end{equation}
Let $R\in{H\b t+1}$ with $U\subseteq R$ and $\mb_R\neq\emptyset$. 
By the definition of $\mb$, $R$ is not a $t$-cover of $\mf_1$. 
Then there exists $F_2\in\mf_1$ with $\dim(R\cap F_2)\le t-1$. 
Note that $F_2$ is a $t$-cover of $\mb_R$.
By \cite[Lemma 2.4]{CAO}, there exists a $(2t+1-\dim(R\cap F_2))$-subspace $R'$ of $V$ such that
\begin{equation*}\label{aa4}
	%\begin{aligned}
|\mb_R|\le{k-\dim(R\cap F_2)\b t-\dim(R\cap F_2)}|\mf_{R'}|
\le{k-t+1\b1}^{t-\dim(R\cap F_2)}{n-2t-1+\dim(R\cap F_2)\b k-2t-1+\dim(R\cap F_2)}.
%\end{aligned}
\end{equation*}
For integer $a\le k-1$, since $n\ge2k-t+1$, we have
$$	\dfrac{{n-a\b k-a}}{{n-a-1\b k-a-1}}=\dfrac{q^{n-a}-1}{q^{k-a}-1}\ge q^{n-k}\ge q^{k-t+1}\ge{k-t+1\b1}.$$
Note that $t-\dim(R\cap F_2)\ge1$. We derive
$$|\mb_R|\le{k-t+1\b1}{n-t-2\b k-t-2}.$$
Then by (\ref{aa1}) and (\ref{aa2}), we obtain
\begin{equation}\label{aa3}
	|\mb|\le{t+1\b1}{k-\alpha\b1}{k-t+1\b1}{n-t-2\b k-t-2}.
\end{equation}
Let $S'$ be a $t$-cover of $\mf_1$ with dimension $t+1$. 
We have $\dim(S\cap S')\ge t$ by \cite[Lemma 2.7]{CAO}.
Then $\dim(U\cap S')\ge t-1$. 
Hence 
$$t\le\dim(S'\cap F_1)\le\dim((S'\cap U)\cap F_1)+\dim S'-\dim(S'\cap U)\le\alpha+2,$$
which implies that $\alpha\ge t-2$. 
Then by $n\ge2k+2t+5$, (\ref{aa3}) and Lemma \ref{gaosifangsuo}, we have
$$\dfrac{|\mb|}{{n-t-1\b k-t-1}}\le\dfrac{q^{k-t-1}-1}{q^{n-t-1}-1}{t+1\b1}{k-t+1\b1}{k-t+2\b1}\le q^{k-3t-1},$$
as desired.	
\end{proof}

For $T\in{V\b t}$, $L,M\in{V\b k+1}$ and $Z\in{V\b t+2}$ with $T\subseteq M\cap L$ and $\dim(M\cap L)\ge t+2$, note that $|\mh_1(T,L,M)|$ and $|\mh_2(Z)|$ are independent of the choice of $M$, $L$, $T$ and $Z$. Write
$$h_1(n,k,t)=|\mh_1(T,L,M)|,\quad h_2(n,k,t)=|\mh_2(Z)|.$$
\begin{lem}\label{bidaxiao}
	Let $n$, $k$ and $t$ be positive integers with  $k\ge t+2$ and $n\ge2k+2t+5$. 
	Then $$h_1(n,k,t)>\dfrac{383}{384}{k-t+1\b1}{n-t-1\b k-t-1},\quad h_2(n,k,t)>\dfrac{1023}{1024}{t+2\b1}{n-t-1\b k-t-1}.$$
\end{lem}
\begin{proof}
	By \cite[Lemma 2.4]{VHMNEW} and 
		$h_2(n,k,t)={t+2\b1}{n-t-1\b k-t-1}-q{t+1\b1}{n-t-2\b k-t-2}$, we have
	\begin{equation*}\label{linshi}
		\begin{aligned}
			\dfrac{h_1(n,k,t)}{{k-t+1\b1}{n-t-1\b k-t-1}}&>1-\dfrac{1}{q^{n-2k+t-1}(q^2-1)}\ge1-\dfrac{1}{q^{3t+4}(q^2-1)}\ge\dfrac{383}{384},\\
			\dfrac{h_2(n,k,t)}{{t+2\b1}{n-t-1\b k-t-1}}&>1-\dfrac{q^{k-t-1}-1}{q^{n-t-1}-1}\ge1-\dfrac{1}{q^{n-k}}\ge1-\dfrac{1}{q^{3t+7}}\ge\dfrac{1023}{1024},
		\end{aligned}
	\end{equation*}
	as desired.
\end{proof}

\begin{lem}\label{t+1,t+2}
	Let $n$, $k$ and $t$ be positive integers with $k\ge t+2$ and $n\ge2k+2t+5$. Suppose $\mf_1,\mf_2\subseteq{V\b k}$ are cross $t$-intersecting families with $\tau_t(\mf_1)\ge\tau_t(\mf_2)\ge t+1$ and $(\tau_t(\mf_1),\tau_t(\mf_2))\neq(t+1,t+1)$. 
	Then  $|\mf_1||\mf_2|<(h_1(n,k,t))^2$. 

\end{lem}
\begin{proof}
		We first show
	\begin{equation}\label{fg}
|\mf_1||\mf_2|\le{t+1\b1}{t+2\b2}{k-t+1\b1}^3{n-t-1\b k-t-1}{n-t-2\b k-t-2}.
	\end{equation} 
	If $\tau_t(\mf_2)=t+1$, then $\tau_t(\mf_1)\ge t+2$. 
	By \cite[Lemma 2.5]{CAO}, we have  
	\begin{equation*}
		\begin{aligned}
		\dfrac{|\mf_1||\mf_2|}{{t+1\b1}{k-t+1\b1}{n-t-1\b k-t-1}}&\le{\tau_t(\mf_1)\b t}{k\b 1}^{\tau_t(\mf_1)-t-2}{k-t+1\b1}^2{n-\tau_t(\mf_1)\b k-\tau_t(\mf_1)}.
		\end{aligned}
	\end{equation*}
	This together with \cite[Lemma 2.3 (\ro2)]{CAO} produces (\ref{fg}). 
    If $\tau_t(\mf_2)\ge t+2$, then $\tau_t(\mf_1)\ge t+2$. We get 
	$$|\mf_1||\mf_2|\le\prod_{i=1}^2{\tau_t(\mf_{i})\b t}{k\b1}^{\tau_t(\mf_{i})-t-2}{k-t+1\b1}^2{n-\tau_t(\mf_{i})\b k-\tau_t(\mf_{i})}$$
	from \cite[Lemma 2.5]{CAO}. This together with \cite[Lemma 2.3 (\ro2)]{CAO} also produces (\ref{fg}).

      By  $n\ge2k+2t+5$, (\ref{fg}) and Lemma \ref{gaosifangsuo}, we have 
      \begin{equation*}
      	\begin{aligned}
      		\dfrac{|\mf_1||\mf_2|}{{k-t+1\b1}^2{n-t-1\b k-t-1}^2}\le\dfrac{q^{k-t-1}-1}{q^{n-t-1}-1}{t+1\b1}{t+2\b2}{k-t+1\b1}
      		\le q^{2k+2t+4-n}<\dfrac{383^2}{384^2}.
      	\end{aligned}
      \end{equation*}
      This together with Lemma \ref{bidaxiao} yields the desired result.
     \end{proof}

\section{Proof of main results}

In this section, we first show Theorem \ref{main1}, and then based on this result prove Theorem \ref{main2}.

\subsection{Proof of Theorem \ref{main1}}

%	Observe that $r=2\le k-t+1$. 
	We may suppose that $\mf_1,\mf_2$ are non-empty cross $t$-intersecting families of $k$-subspaces of $V$  with $\tau_t(\mf_i)\ge t+1$ for each $i$, and $|\mf_1||\mf_2|$ takes the maximum value. Observe that 
	\begin{equation}\label{md}
	|\mf_1||\mf_2|\ge\max_{i\in\{1,2\}}(h_i(n,k,t))^2.
\end{equation}
We claim that $\tau_t(\mf_1)=\tau_t(\mf_2)=t+1$. Indeed, if $(\tau_t(\mf_1),\tau_t(\mf_2))\neq(t+1,t+1)$, then $k\ge t+2$. 
By Lemma \ref{t+1,t+2}, we have $|\mf_1||\mf_2|<(h_1(n,k,t))^2$. 
		This contradicts (\ref{md}).

		Denote the set of all $t$-covers of $\mf_i$ with dimension $t+1$ by $\mt_i$ for each $i$. 
	%	For $i\in\{1,2\}$, let $\mt_i$ be the set of all $t$-covers of $\mf_i$ with dimension $t+1$.
	%	Suppose that $S_0\in(\mt_i)_A$ and $F_0\in\mf_i\bs(\mf_i)_A$, where $A\in{V\b t}$. 
	%	We have $\dim(A\cap F_0)=t-1$ and $\dim(S_0\cap F_0)=t$. Then 	$k+1\le\dim(A+F_0)\le\dim(S_0+F_0)\le k+1$, from which we get
	%	\begin{equation}\label{xingz}
	%	A+F_0=S_0+F_0,\quad\dim(A+F_0)=k+1.
	%		\end{equation}
		%Hence $A+F_0=M_i+F_0$ and $(\mt_i)_A\subseteq\left\{T\in{M_i\b t+1}: A\subseteq T\right\}$, which imply that
	%	\begin{equation}\label{ta}
		%	t+1\le\dim M_i\le k+1,\quad|(\mt_i)_A|\le{\dim M_i-t\b1}.
		%\end{equation}	
Write 
\begin{equation}\label{fj}
	\begin{aligned}
		\ma_1&=\left\{F\in\mf_1: \exists S\in\mt_2,\ S\subseteq F\right\},\quad\mb_1=\{F\in\mf_1: \forall S\in\mt_2,\ S\not\subseteq F\}.\\
		\ma_2&=\left\{F\in\mf_2: \exists S\in\mt_1,\ S\subseteq F\right\},\quad\mb_2=\{F\in\mf_2: \forall S\in\mt_1,\ S\not\subseteq F\}.
	\end{aligned}
\end{equation}
Observe that $\mf_1=\ma_1\cup\mb_1$ and $\mf_2=\ma_2\cup\mb_2$. 
On the other hand, by \cite[Lemma 2.7]{CAO}, $\mt_1$ and $\mt_2$ are cross $t$-intersecting, 
implying that $t\le\tau_t(\mt_i)\le t+1$ for each $i$. 
In the following, we investigate $\mf_1$ and $\mf_2$ in two cases.

	\medskip
		\noindent{{\bf Case 1.} $(\tau_t(\mt_1),\tau_t(\mt_2))\neq(t+1,t+1)$.}
		\medskip
		
		If $k=t+1$, then by the maximality of $\mf_1$ and $\mf_2$, we have $\mt_2\subseteq\mf_1\subseteq\mt_2$ and $\mt_1\subseteq\mf_2\subseteq\mt_1$, implying that $\mf_1=\mt_2$ and $\mf_2=\mt_1$. 
		So 
		$\tau_t(\mt_1)=\tau_t(\mt_2)=t+1$. 
		This contradicts the assumption that $(\tau_t(\mt_1),\tau_t(\mt_2))\neq(t+1,t+1)$. 
		Consequently $k\ge t+2$. 
		
			For $i\in\{1,2\}$ and $A\in{V\b t}$ with $(\mt_i)_A\neq\emptyset$, let $E$ be the space spanned by the union of all memebers of $(\mt_i)_A$.
		For $S_0\in(\mt_i)_A$ and $F_0\in\mf_i\bs(\mf_i)_A$, we have
		$\dim(S_0\cap F_0)=t$ and $\dim(A\cap F_0)=t-1$. 
		Then 	$k+1=\dim(A+F_0)\le\dim(S_0+F_0)=k+1.$ We further get
		\begin{equation}\label{xingz}
			A+F_0=S_0+F_0=E+F_0,\quad\dim(A+F_0)=k+1.
		\end{equation}
		 Therefore
		\begin{equation}\label{ta}
			t+1\le\dim E\le k+1,\quad|(\mt_i)_A|\le{\dim E-t\b1}\le{k-t+1\b1}.
		\end{equation}	
		
		%\medskip
		\noindent{{\bf Case 1.1.} $(\tau_t(\mt_1),\tau_t(\mt_2))=(t,t)$.}
	\medskip
	
	Let $M$ and $L$ be the spaces spanned by the union of all members of $\mt_1$ and $\mt_2$, respectively, and $m=\dim M$, $\ell=\dim L$. 
	By (\ref{ta}), we have $m,\ell\in\{t+1,t+2,\dots,k+1\}$.

	Suppose $(m,\ell)\neq(k+1,k+1)$. 
	We may assume  $m\le k+1$ and $\ell\le k$. 
	Since $\tau_t(\mt_2)=t$, by (\ref{ta}), we obtain  $|\mt_2|\le{\ell-t\b1}\le{k-t\b1}$. 
	Then $$|\ma_1|\le{k-t\b 1}{n-t-1\b k-t-1}.$$ 
	This together with (\ref{fj}) and Lemma \ref{t+1} produces
	$$|\mf_1|\le\left({k-t\b1}+q^{k-3t-1}\right){n-t-1\b k-t-1}.$$
	Similarly, we derive
	$$|\mf_2|\le\left({k-t+1\b1}+q^{k-3t-1}\right){n-t-1\b k-t-1}.$$
Then by Lemma \ref{gaosifangsuo}, we have
\begin{equation*}
	\begin{aligned}
		\dfrac{|\mf_1||\mf_2|}{{k-t+1\b1}^2{n-t-1\b k-t-1}^2}
		\le&\ \dfrac{q^{k-t}-1+q^{k-3t-1}(q-1)}{q^{k-t+1}-1}\cdot\left(1+\dfrac{q^{k-3t-1}(q-1)}{q^{k-t+1}-1}\right)\\
		\le&\left(\dfrac{1}{q}+\dfrac{1}{q^{2t+1}}\right)\left(1+\dfrac{1}{q^{2t+1}}\right)\le\dfrac{5}{8}\cdot\dfrac{9}{8}<\dfrac{383^2}{384^2}.
	\end{aligned}
\end{equation*}
This together with  Lemma \ref{bidaxiao} produces $|\mf_1||\mf_2|<(h_1(n,k,t))^2$, which contradicts (\ref{md}). 
Therefore $m=\ell=k+1$.

Suppose  that $t$-subspaces $T_1$ and $T_2$ are $t$-covers of $\mt_1$ and $\mt_2$, respectively. 
Assume $T_1\neq T_2$. 
Observe that $|\mt_1|\ge2$. 
Then $T_2\not\subseteq S_1$ for some $S_1\in\mt_1$. 
Note that $\mt_1$ and $\mt_2$ are cross $t$-intersecting. 
It is routine to check that $\dim(T_2\cap S_1)=t-1$. 
For $S_2\in\mt_2$, since $\dim(S_2\cap S_1)\ge t$, we have $S_2\subseteq T_2+S_1$. Then $k+1=\ell\le\dim(T_2+S_1)=t+2$, which contradicts the fact that $k\ge t+2$. 
Thus $T_1=T_2\in{L\cap M\b t}$.

Let $F_1\in\mf_1\bs(\mf_1)_{T_1}$ and $F_2\in\mf_2\bs(\mf_2)_{T_1}$. 
We have $\dim(F_1\cap M)=\dim(F_2\cap L)=k$ by (\ref{xingz}). 
Therefore 
$$\mf_1\bs(\mf_1)_{T_1}\subseteq\left\{F\in{M\b k}: T_1\not\subseteq F\right\},\quad\mf_2\bs(\mf_2)_{T_1}\subseteq\left\{F\in{L\b k}: T_1\not\subseteq F\right\}.$$  
For $F_3\in(\mf_1)_{T_1}$, by $\dim(F_2\cap F_3)\ge t$ and $T_1\not\subseteq F_2\subseteq L$, we have $\dim(F_3\cap L)\ge t+1$. 
So
$$(\mf_1)_{T_1}\subseteq\left\{F\in{V\b k}: T_1\subseteq F,\ \dim(F\cap L)\ge t+1\right\}.$$
Similarly, we derive that
$$(\mf_2)_{T_1}\subseteq\left\{F\in{V\b k}: T_1\subseteq F,\ \dim(F\cap M)\ge t+1\right\}.$$
Suppose that $\dim(L\cap M)\le t+1$. 
Then $\{F\in{M\b k}: T_1\not\subseteq F\}$ and $\{F\in{L\b k}: T_1\not\subseteq F\}$ are not cross $t$-intersecting.
Consequently for some $i\in\{1,2\}$, 
$$|\mf_i\bs(\mf_i)_{T_1}|<\left|\left\{F\in{M\b k}: T_1\not\subseteq F\right\}\right|.$$ 
Then $|\mf_1||\mf_2|<(h_1(n,k,t))^2$, which contradicts (\ref{md}). 
So $\dim(L\cap M)\ge t+2$. 

Notice that $\mh_1(T_1,L,M)$, $\mh_1(T_1,M,L)$ are cross $t$-intersecting and $\mf_1\subseteq\mh_1(T_1,L,M)$, $\mf_2\subseteq\mh_1(T_1,M,L)$.  
By the maximality of $\mf_1$ and $\mf_2$, 
we have 
$\mf_1=\mh_1(T_1,L,M)$ and $\mf_2=\mh_1(T_1,M,L)$.

%Observe that there exist $S_1\in\mt_1$ and $S_2\in\mt_2$ such that $T_1\not\subseteq S_2$ and $T_2\not\subseteq S_1$. 

	\medskip
		\noindent{{\bf Case 1.2.} $(\tau_t(\mt_1),\tau_t(\mt_2))\neq(t,t)$.}
		\medskip

W.l.o.g., assume that $(\tau_t(\mt_1),\tau_t(\mt_2))=(t+1,t)$. 
Suppose $\mt_2=\{S\}$. 
Since $\mt_1$ and $\mt_2$ are cross $t$-intersecting, we have $\mt_1=\bigcup_{T\in{S\b t}}(\mt_1)_T$. 
For each $T\in{S\b t}$,  we have $|(\mt_1)_T|\le{k-t+1\b1}$ by (\ref{ta}). 
Then $|\mt_1|\le{t+1\b1}{k-t+1\b1}$. 
We further get $$|\ma_2|\le{t+1\b1}{k-t+1\b1}{n-t-1\b k-t-1}.$$ 
This together with (\ref{fj}) and Lemma \ref{t+1} yields
	$$|\mf_2|\le\left({t+1\b1}{k-t+1\b1}+q^{k-3t-1}\right){n-t-1\b k-t-1}.$$
Since $|\mt_2|=1$, we have $|\ma_1|\le{n-t-1\b k-t-1}$. This together with (\ref{fj}) and Lemma \ref{t+1} produces
$$|\mf_1|\le\left(1+q^{k-3t-1}\right){n-t-1\b k-t-1}.$$
Then by Lemma \ref{gaosifangsuo}, we have
	\begin{equation}\label{a1}
	\begin{aligned}
		\dfrac{|\mf_1||\mf_2|}{{k-t+1\b1}^2{n-t-1\b k-t-1}^2}
		\le&\left(1+q^{k-3t-1}\right)\left(\dfrac{q^{t+1}-1}{q^{k-t+1}-1}+\dfrac{q^{k-3t-1}(q-1)^2}{(q^{k-t+1}-1)^2}\right)\\
		<&\left(1+q^{k-3t-1}\right)\left(\dfrac{1}{q^{k-2t}}+\dfrac{1}{q^{k+t+1}}\right)\\
		=&\ \dfrac{1}{q^{k-2t}}+\dfrac{1}{q^{t+1}}+\dfrac{1}{q^{k+t+1}}+\dfrac{1}{q^{4t+2}}\le\dfrac{51}{64}<\dfrac{383^2}{384^2}
	\end{aligned}
\end{equation}
	for $k\ge2t+1$, and 
\begin{equation}\label{a2}
	\begin{aligned}
		\dfrac{|\mf_1||\mf_2|}{{t+2\b1}^2{n-t-1\b k-t-1}^2}
		\le&\left(1+\dfrac{1}{q^{t+1}}\right)\left(\dfrac{(q^{t+1}-1)^2}{(q^{t+2}-1)^2}+\dfrac{(q-1)^2}{q^{t+1}(q^{t+2}-1)^2}\right)\\
		\le&\left(1+\dfrac{1}{q^{t+1}}\right)\left(\dfrac{1}{q^2}+\dfrac{1}{q^{3t+3}}\right)\le\dfrac{85}{256}<\dfrac{1023^2}{1024^2}
	\end{aligned}
\end{equation}
for $k\le2t$. 
By (\ref{a1}), (\ref{a2}) and Lemma \ref{bidaxiao}, we get $|\mf_1||\mf_2|<\max_{i\in\{1,2\}}(h_i(n,k,t))^2$. 
This contradicts (\ref{md}). So $|\mt_2|\ge2$. 

Notice that $\tau_t(\mt_2)=t$. Let $S_1$ and $S_2$ be members of $\mt_2$ with $\dim(S_1\cap S_2)=t$.
Since $\mt_1$, $\mt_2$ are cross $t$-intersecting and $\tau_t(\mt_1)=t+1$, 
there exists $U\in\mt_1$ with $S_1\cap S_2\not\subseteq U$.  
By $\dim(U\cap S_1)\ge t$ and $\dim(U\cap S_2)\ge t$, we have $\dim(U\cap(S_1\cap S_2))=t-1$ and $U\subseteq S_1+S_2$. 
Thus 
$$\mt_1\bs(\mt_1)_{S_1\cap S_2}\subseteq\left\{U\in{S_1+S_2\b t+1}: S_1\cap S_2\not\subseteq U\right\}.$$
Notice that $\dim(S_1+S_2)=t+2$. 
We obtain $$|\mt_1\bs(\mt_1)_{S_1\cap S_2}|\le{t+2\b1}-{2\b1}=q^2{t\b1}.$$
By (\ref{ta}), we have $|(\mt_1)_{S_1\cap S_2}|\le{k-t+1\b1}$. 
Therefore 
$$|\mt_1|=|(\mt_1)_{S_1\cap S_2}|+|\mt_1\bs(\mt_1)_{S_1\cap S_2}|\le{k-t+1\b1}+q^2{t\b1}.$$
This together with (\ref{fj}) and Lemma \ref{t+1} produces
$$|\mf_2|\le\left({k-t+1\b1}+q^2{t\b1}+q^{k-3t-1}\right){n-t-1\b k-t-1}.$$
For each $S_3\in\mt_2$, since $S_1\cap S_2\subseteq S_3$, $\dim(S_3\cap U)\ge t$ and $S_1\cap S_2\not\subseteq U$,  
we have $S_3\subseteq(S_1\cap S_2)+U$. Observe that $\dim((S_1\cap S_2)+U)=t+2$. We further get $|\mt_2|\le{2\b1}$. 
This together with (\ref{fj}) and Lemma \ref{t+1} yields
\begin{equation*}
	\begin{aligned}
|\mf_1|\le\left({2\b1}+q^{k-3t-1}\right){n-t-1\b k-t-1}.
	\end{aligned}
\end{equation*}
Then by Lemma \ref{gaosifangsuo} and $\frac{q^2-1}{q^{t+2}-1}\le\frac{q^2-1}{q^3-1}\le\frac{3}{7}$, we have
\begin{equation}\label{b1}
	\begin{aligned}
	\dfrac{|\mf_1||\mf_2|}{{k-t+1\b1}^2{n-t-1\b k-t-1}^2}
	\le&\left(\dfrac{q^2-1}{q^{k-t+1}-1}+\dfrac{1}{q^{2t+1}}\right)\left(1+\dfrac{q^2(q^t-1)}{q^{k-t+1}-1}+\dfrac{1}{q^{2t+1}}\right)\\
	\le&\left(\dfrac{q^2-1}{q^{t+3}-1}+\dfrac{1}{q^{2t+1}}\right)\left(1+\dfrac{q^2(q^t-1)}{q^{t+3}-1}+\dfrac{1}{q^{2t+1}}\right)\\
		\le&\left(\dfrac{1}{q^{t+1}}+\dfrac{1}{q^{2t+1}}\right)\left(1+\dfrac{1}{q}+\dfrac{1}{q^{2t+1}}\right)
	\le\dfrac{39}{64}<\dfrac{383^2}{384^2}
	\end{aligned}
\end{equation}
for $k\ge 2t+2$, and 
\begin{equation}\label{b2}
	\begin{aligned}
		\dfrac{|\mf_1||\mf_2|}{{t+2\b1}^2{n-t-1\b k-t-1}^2}
		\le&\left(\dfrac{q^2-1}{q^{t+2}-1}+\dfrac{1}{q^{2t+1}}\right)\left(1+\dfrac{q^2(q^t-1)}{q^{t+2}-1}+\dfrac{1}{q^{2t+1}}\right)\\
		\le&\left(\dfrac{q^2-1}{q^{t+2}-1}+\dfrac{1}{8}\right)\left(\dfrac{17}{8}-\dfrac{q^2-1}{q^{t+2}-1}\right)\le\dfrac{2945}{3136}<\dfrac{1023^2}{1024^2}
	\end{aligned}
\end{equation}
for $k\le2t+1$. 
By (\ref{b1}), (\ref{b2}) and Lemma \ref{bidaxiao}, we obtain $|\mf_1||\mf_2|<\max_{i\in\{1,2\}}(h_i(n,k,t))^2$. 
This contradicts (\ref{md}).

\medskip

\noindent{{\bf Case 2.} $(\tau_t(\mt_1),\tau_t(\mt_2))=(t+1,t+1)$.}

\medskip
\noindent{{\bf Case 2.1.} $\mt_1$ is $t$-intersecting.}
\medskip

By assumption and \cite[Remark (\ro2) in Section 9.3]{BCN}, there exists $Z\in{V\b t+2}$ such that $\mt_1\subseteq{Z\b t+1}$. 
If there exists $S\in\mt_2$ with $\dim(S\cap Z)\le t$, 
then $$\dim(S'\cap S)=\dim(S'\cap(S\cap Z))\le t-1$$ for some $S'\in\mt_1$. This contradicts the fact that $\mt_1$ and $\mt_2$ are cross $t$-intersecting. Therefore, each member of $\mt_2$ is also a subspace of $Z$, i.e.,  $\mt_2\subseteq{Z\b t+1}$.

If $\dim(H\cap Z)=t$ for some $H\in\mf_1$, then each member of $\mt_1$ contains $H\cap Z$. This contradicts $\tau_t(\mt_1)=t+1$. 
So $\dim(F\cap Z)\ge t+1$ for each $F\in\mf_1$, i.e., $\mf_1\subseteq\mh_2(Z)$. 
Similarly, we have $\mf_2\subseteq\mh_2(Z)$. 
Notice that  $\mh_2(Z)$ is $t$-intersecting. 
By the maximality of $\mf_1$ and $\mf_2$, we have $\mf_1=\mf_2=\mh_2(Z)$.

\medskip
\noindent{{\bf Case 2.2.} $\mt_1$ is not $t$-intersecting.}
\medskip

We first show  upper bounds of $|\mt_1|$ and $|\mt_2|$. 
Let $U_1,U_2\in\mt_1$ with $\dim(U_1\cap U_2)<t$, and $S\in\mt_2$. 
Note that $\dim(U_i\cap S)\ge t$ for each $i$. We have
$$t+1=\dim S\ge \dim(U_1\cap S)+\dim(U_2\cap S)-\dim((U_1\cap U_2)\cap S)\ge2t-\dim((U_1\cap U_2)\cap S).$$
Hence $$t-1\le\dim((U_1\cap U_2)\cap S)\le\dim(U_1\cap U_2)\le t-1,$$ 
which implies that 
$$\dim(U_1\cap U_2)=t-1,\quad U_1\cap U_2=(U_1\cap U_2)\cap S\subseteq S.$$
We further obtain $\dim(U_1\cap S)=\dim(U_2\cap S)=t$ and $S=(U_1\cap S)+(U_2\cap S)$. 
So $$\mt_2\subseteq\left\{W_1+W_2: W_i\in{U_i\b t},\ U_1\cap U_2\subseteq W_i,\ i=1,2\right\}.$$ 
Then $$|\mt_2|\le{(t+1)-(t-1)\b1}^2={2\b1}^2.$$
If $\mt_2$ is $t$-intersecting, then by discussion in Case 2.1, $\mt_1$ is also $t$-intersecting, a contradiction.  
Thus $\mt_2$ is not $t$-intersecting. 
Similarly, we derive $|\mt_1|\le{2\b1}^2$.

Assume that $t=1$ and $|\mt_1|\ge3$. Note that $U_1\cap U_2=\{0\}$. 
Since $\mt_1$ and $\mt_2$ are cross $1$-intersecting, we have $\mt_2\subseteq{U_1+U_2\b2}$. 
Observe that there exist $Y_1,Y_2\in\mt_2$ with $Y_1\cap Y_2=\{0\}$. 
Then $\mt_1\subseteq{Y_1+Y_2\b2}\subseteq{U_1+U_2\b2}$.  

Suppose that there exists $U_3\in\mt_1$ with $U_3\cap U_1=U_3\cap U_2=\{0\}$, 
and $U_i$ is spanned by vectors $\alpha_i$ and $\beta_i$ ($i=1,2$). 
Note that $U_3\in{U_1+U_2\b2}$. 
Then $U_3$ is spanned by vectors $a_i\alpha_1+b_i\beta_1+c_i\alpha_2+d_i\beta_2$, where $a_i,b_i,c_i,d_i\in\ff$ ($i=1,2$). 
If $a_1b_2-a_2b_1=0$, then $\dim(U_3\cap U_2)\ge1$, a contradiction. 
Thus $a_1b_2-a_2b_1\neq0$.
%Similarly, we derive that $c_1d_2-c_2d_1\neq0$. 
We may assume $a_1=b_2=1$ and $a_2=b_1=0$. 
Recall that $S\in\mt_2$. 
Since $\dim(S\cap U_1)=\dim(S\cap U_2)=1$, 
 $S$ is spanned by vectors $r_1\alpha_1+s_1\beta_1$ and $r_2\alpha_2+s_2\beta_2$ for some  $(r_1,s_1),(r_2,s_2)\in\ff^2\bs\{(0,0)\}$. 
Note that $\dim(S\cap U_3)\ge1$. 
There exists $u\in\ff^*$ such that $$r_1\alpha_1+s_1\beta_1+u(r_2\alpha_2+s_2\beta_2)=r_1(\alpha_1+c_1\alpha_2+d_1\beta_2)+s_1(\beta_1+c_2\alpha_2+d_2\beta_2)\in S\cap U_3.$$ 
We further get $r_2=u^{-1}(r_1c_1+s_1c_2)$ and $s_2=u^{-1}(r_1d_1+s_1d_2)$, which imply that $S$ is spanned by $r_1\alpha_1+s_1\beta_1$ and $(r_1c_1+s_1c_2)\alpha_2+(r_1d_1+s_1d_2)\beta_2$. 
Hence 
$$|\mt_2|\le|\ff^2\bs\{(0,0)\}|=q^2-1.$$

%$|\mt_2|$ is no more than the size of $\ff^2\bs\{(0,0)\}$, i.e., $|\mt_2|\le q^2-1$. 
%	Then
%$$(|\mt_1|+1)(|\mt_2|+1)\le q^2\left({2\b1}^2+1\right)<{3\b1}^2,$$
%	as desired.

Next, w.l.o.g., suppose $\dim(U_4\cap U_1)=1$ for some $U_4\in\mt_1$. 
Note that $\dim(U_1+U_4)=3$. 
This together with $U_4\subseteq U_1+U_2$ and $U_1\cap U_2=\{0\}$ yields $\dim(U_2\cap(U_1+U_4))=1$. 
If $U_1\cap U_4\not\subseteq S$, then $S\subseteq U_1+U_4$ and $S=W+(U_2\cap(U_1+U_4))$ for some $W\in{U_1\b 1}\bs\{U_1\cap U_4\}$. 
Consequently $$|\mt_2\bs(\mt_2)_{U_1\cap U_4}|\le{2\b1}-1=q.$$ 
If $U_1\cap U_4\subseteq S$, then by $\dim(S\cap U_2)=1$, we obtain $|(\mt_2)_{U_1\cap U_4}|\le{2\b1}$. 
Therefore $$|\mt_2|=|(\mt_2)_{U_1\cap U_4}|+|\mt_2\bs(\mt_2)_{U_1\cap U_4}|\le q+{2\b1}=2q+1.$$
Now we have $|\mt_2|\le\max\{q^2-1,2q+1\}$. 
Similarly, since $\mt_2$ is not $t$-initersecting, when $t=1$ and $|\mt_2|\ge3$, we have $|\mt_1|\le\max\{q^2-1,2q+1\}$.  

In a word, for each $i$, we have
\begin{equation}\label{t1t2}
	|\mt_i|\le\left\{
	\begin{aligned}
		&\max\{q^2-1,2q+1\}\le{2\b1}^2,& &\mbox{if\ }t=1\ \mbox{and\ }|\mt_{3-i}|\ge3,\\
		&{2\b1}^2,& &\mbox{otherwise}.
	\end{aligned}
	\right.
\end{equation}
If $k=t+1$, we have known that $\mf_1=\mt_2$ and $\mf_2=\mt_1$. Then by (\ref{t1t2}),
\begin{equation*}
	|\mf_1||\mf_2|\le\left\{
	\begin{aligned}
		&\max\left\{(q^2-1)^2,(2q+1)^2,2{2\b1}^2\right\}<(h_2(n,2,1))^2,& &\mbox{if\ }t=1,\\
		&{2\b1}^4<{4\b1}^2\le{t+2\b1}^2=(h_2(n,t+1,t))^2,& &\mbox{if\ }t\ge2.
	\end{aligned}
	\right.
\end{equation*}
This contradicts (\ref{md}). So $k\ge t+2$. 

By (\ref{fj}) and Lemma \ref{t+1}, we have
\begin{equation}\label{f1f2}
	\begin{aligned}
		|\mf_1|\le\left(|\mt_2|+q^{k-3t-1}\right){n-t-1\b k-t-1},\quad|\mf_2|\le\left(|\mt_1|+q^{k-3t-1}\right){n-t-1\b k-t-1}.
	\end{aligned}
\end{equation}
Assume $k\le2t+1$. Then
\begin{equation*}
	\begin{aligned}
		\dfrac{|\mf_1||\mf_2|}{{n-t-1\b k-t-1}^2}
		\le|\mt_1||\mt_2|+q^{-t}(|\mt_1|+|\mt_2|)+q^{-2t}.
	\end{aligned}
\end{equation*}
Suppose $t=1$ and $|\mt_1|,|\mt_2|\ge3$.
If $q=2$, then by  (\ref{t1t2}) and Lemma \ref{bidaxiao}, we have
\begin{equation}\label{c1}
	%\begin{aligned}
		\dfrac{(h_2(n,k,1))^2-|\mf_1||\mf_2|}{{n-2\b k-2}^2}\ge\dfrac{1023^2}{1024^2}{3\b1}^2-5^2-5-\dfrac{1}{2^{2}}>0.
	%\end{aligned}
\end{equation}
If $q\ge3$, then by (\ref{t1t2}), $n\ge2k+2t+5$, Lemma \ref{gaosifangsuo} and
$$h_2(n,k,1)={3\b1}{n-2\b k-2}-q{2\b1}{n-3\b k-3},$$ we obtain
\begin{equation}\label{c2}
	\begin{aligned}
		\dfrac{(h_2(n,k,1))^2-|\mf_1||\mf_2|}{{n-2\b k-2}^2}
		\ge&\left({3\b1}-\dfrac{q(q^{k-2}-1){2\b1}}{q^{n-2}-1}\right)^2-(q^2-1)^2-\dfrac{2(q^2-1)}{q}-\dfrac{1}{q^2}\\
		\ge&\left({3\b1}-\dfrac{q(q+1)}{q^{10}}\right)^2-\left(q^2-1+\dfrac{1}{q^2}\right)^2>0.
	\end{aligned}
\end{equation}
Suppose $t=1$ and $|\mt_i|\le2$ for some $i\in\{1,2\}$. By $\frac{q^2-1}{q^3-1}\le\frac{3}{7}$, (\ref{t1t2}) and Lemma \ref{gaosifangsuo}, we get
\begin{equation}\label{c3}
	\begin{aligned}
		\dfrac{|\mf_1||\mf_2|}{{3\b1}^2{n-2\b k-2}^2}\le&\dfrac{2{2\b1}^2}{{3\b1}^2}+\dfrac{2+{2\b1}^2}{q{3\b1}^2}+\dfrac{1}{q^2{3\b1}^2}
		\le\dfrac{18}{49}+\dfrac{1}{49}+\dfrac{9}{98}+\dfrac{1}{196}<\dfrac{1023^2}{1024^2}.
	\end{aligned}
\end{equation}
Suppose $t\ge2$, then by  (\ref{t1t2}) and Lemma \ref{gaosifangsuo}  , we have
\begin{equation}\label{c4}
	\begin{aligned}
		\dfrac{|\mf_1||\mf_2|}{{t+2\b1}^2{n-t-1\b k-t-1}^2}\le&\left(\dfrac{{2\b1}^2}{{t+2\b1}}\right)^2+\dfrac{2}{q^t}\left(\dfrac{{2\b1}}{{t+2\b1}}\right)^2+\dfrac{1}{(q^{t}{t+2\b1})^2}\\
		\le&\ \dfrac{(q^2-1)^4}{(q-1)^2(q^{4}-1)^2}+\dfrac{2(q^2-1)^2}{q^t(q^{t+2}-1)^2}+\dfrac{(q-1)^2}{q^{2t}(q^{t+2}-1)^2}\\
		\le&\ \dfrac{(q+1)^2}{(q^{2}+1)^2}+\dfrac{1}{q^{3t-1}}+\dfrac{1}{q^{4t+2}}\le\ \dfrac{10041}{25600}<\ \dfrac{1023^2}{1024^2}.
	\end{aligned}
\end{equation}
Now assume $k\ge2t+2$. Note that $|\mt_1||\mt_2|\le{t+2\b1}^2$ by (\ref{t1t2}). This together with  (\ref{t1t2}), (\ref{f1f2}) and Lemma \ref{gaosifangsuo} yields
\begin{equation}\label{c5}
	\begin{aligned}
		\dfrac{|\mf_1||\mf_2|}{{k-t+1\b1}^2{n-t-1\b k-t-1}^2}
		\le&\ \dfrac{|\mt_1||\mt_2|+q^{k-3t-1}(|\mt_1|+|\mt_2|)+q^{2k-6t-2}}{{k-t+1\b1}^2}\\
		\le&\ \dfrac{(q^{t+2}-1)^2}{(q^{t+3}-1)^2}+\dfrac{2q^{k-3t-1}(q^2-1)^2}{(q^{k-t+1}-1)^2}+\dfrac{q^{2k-6t-2}(q-1)^2}{(q^{k-t+1}-1)^2}\\
		\le&\ \dfrac{1}{q^2}+\dfrac{1}{q^{k+t-2}}+\dfrac{1}{q^{4t+2}}\le\dfrac{25}{64}<\dfrac{383^2}{384^2}.
	\end{aligned}
\end{equation}
By (\ref{c1})--(\ref{c5}) and Lemma \ref{bidaxiao}, we obtain $|\mf_1||\mf_2|<\max_{i\in\{1,2\}}(h_i(n,k,t))^2$. 
This contradicts (\ref{md}). 

In summary, assuming $k=t+1$, we have $\mf_1=\mf_2=\mh_2(Z)$ for some $Z\in{V\b t+2}$; assuming $k\ge t+2$, we have 
$\mf_1=\mh_1(T,L,M)$ and $\mf_2=\mh_1(T,M,L)$ for some $T\in{V\b t}$ and $L,M\in{V\b k+1}$ with $T\subseteq L\cap M$ and $\dim(L\cap M)\ge t+2$, or $\mf_1=\mf_2=\mh_2(Z)$ for some $Z\in{V\b t+2}$.

When $k\ge t+2$, by \cite[Lemmas 2.6-2.9]{VHM2}, we have 
$(h_1(n,k,t))^2>(h_2(n,k,t))^2$ if $k>2t+1$; $(h_1(n,k,t))^2\le (h_2(n,k,t))^2$ if $k\le2t+1$, and equality holds if and only if $(k,t)=(3,1)$. 
 This finishes the proof of Theorem \ref{main1}.

\subsection{Proof of Theorem \ref{main2}}

We may suppose that $\mf_1$, $\mf_2$, \dots, $\mf_r$ are non-empty $r$-cross $t$-intersecting families of $k$-subspaces of $V$ with $\tau_t(\mf_i)\ge t+1$ for each $i$, and $\prod_{i=1}^r|\mf_i|$ takes the maximum value. 
	
	Let $d\in\{1,2,\dots,r\}$. Assume that there exists a subspace $F_0$ of $V$ with $\dim(\bigcap_{j=0}^dF_j)\ge t$ for any $F_i\in\mf_i$,  $1\le i\le d$. 
	For $m\in\{0,1,2,\dots,d-1\}$, suppose that there exist $F'_1\in\mf_1$, \dots, $F'_m\in\mf_m$ such that
	 $\dim F_0\ge\dim(\bigcap_{j=0}^mF'_j)+m\ge t+m$, where $F'_0=F_0$. 
	Since $\tau_t(\mf_{m+1})\ge t+1$, there exists $F'_{{m+1}}\in\mf_{{m+1}}$ not containing $\bigcap_{j=0}^mF'_{{j}}$, 
	which implies that 
	$$\dim F_{{0}}\ge\dim\left(\bigcap_{j=0}^{m}F'_{{j}}\right)+m\ge\dim\left(\bigcap_{j=0}^{m+1}F'_{{j}}\right)+1+m\ge t+m+1.$$
	Note that $\dim F_{{0}}\ge t$. By induction, we get $\dim F_{0}\ge t+d$.

	Set $d=r-2$ and $F_0=A\cap B$, where $A\in\mf_{r-1}$ and $B\in\mf_{r}$. Then $\dim(A\cap B)\ge t+r-2$.  
	Hence $\mf_{r-1}$ and $\mf_r$ are cross $(t+r-2)$-intersecting. 
	We further conclude that $\mf_i$ and $\mf_j$ are cross $(t+r-2)$-intersecting for distinct $i,j$. 
	
	 If $\tau_{t+r-2}(\mf_{\ell})=t+r-2$ for some $\ell$, then there exists a $t$-subspace contained in each member of $\mf_\ell$, a contradiction to $\tau_t(\mf_\ell)\ge t+1$. 
	Thus $\tau_{t+r-2}(\mf_i)\ge t+r-1$ for each $i$. 
	For distinct $i,j\in\{1,2,\dots,r\}$, since $k\ge t+r-1$ and $n\ge2k+2t+2r+1$, by Theorem \ref{main1}, we get
	\begin{equation}\label{ij}
	|\mf_i||\mf_j|\le\max\{(h_1(n,k,t+r-2))^2,(h_2(n,k,t+r-2))^2\},
	\end{equation} which implies that 
$$\prod_{i=1}^r|\mf_i|\le\max\{(h_1(n,k,t+r-2))^r,(h_2(n,k,t+r-2))^r\}.$$  
We may suppose the equality holds. 
Then the equality in (\ref{ij}) holds for distinct $i,j$. 

Suppose $k>2t+2r-3$. By Theorem \ref{main1}, we have $\mf_1=\mh_1(T,L,M)$, then $\mf_i=\mh_1(T,M,L)$ ($i\in\{2,3,\dots,r\}$) 
for some $T\in{V\b t}$ and $M,L\in{V\b k+1}$ with $T\subseteq L\cap M$ and $\dim(L\cap M)\ge t+r$. 
  Note that $\mf_2$ and $\mf_3$ are cross $(t+r-2)$-intersecting. Then $\mh_1(T,M,L)$ is $(t+r-2)$-intersecting. 
If $p=\dim(M\cap L)\le k$, then since $n\ge2k+2t+2r+1$, there exists $Y\in\mh_1(T,M,L)$ with 
$$T\subseteq Y,\quad\dim(Y\cap M)=(t+r-2)+(k+1-p),\quad Y\cap L=T.$$ 
For $Y'\in{L\b k}$ with $T\not\subseteq Y'$, we have $\dim(Y\cap Y')=\dim((Y\cap L)\cap Y')<t+r-2$. 
This contradicts the fact that $\mh_1(T,M,L)$ is $(t+r-2)$-intersecting. 
So $M=L$, implying that $\mf_i=\mh_1(T,M,M)$ for each $i$. 
On the other hand, if $k\le2t+2r-3$, then by Theorem \ref{main1}, we have $\mf_1=\mf_i=\mh_2(Z)$ ($i\in\{2,3,\dots,r\}$)  for some $Z\in{V\b t+r}$.

To finish our proof, it is sufficient to show that neither $\dim(\bigcap_{i=1}^rA_i)$ nor $\dim(\bigcap_{i=1}^rB_i)$ is less than $t$ for any $A_i\in\mh_1(T,M,M)$ and $B_i\in\mh_2(Z)$, $1\le i\le r$. By \cite[Lemma 4.7]{CAO}, the desired result follows.

\medskip
\noindent{\bf Remark.} Set $d=r-1$ and $F_0\in\mf_r$. Then $k=\dim F_0\ge t+r-1$. So $r>k-t+1$ implies that there does not exist non-empty $r$-cross $t$-intersecting families of $k$-subspaces of $V$.

\medskip
	
\noindent{\bf Acknowledgment.}	
K. Wang is supported by the National Key R\&D Program of China (No. 2020YFA0712900) and National Natural Science Foundation of China (12071039, 12131011).


\begin{thebibliography}{99}
		\bibitem{SHMT}R. Ahlswede and L.H. Khachatrian, The complete nontrivial-intersection theorem for systems of finite sets, J. Combin. Theory Ser. A 76 (1996) 121--138.
		\bibitem{SEKRT} R. Ahlswede and L.H. Khachatrian, The complete intersection theorem for systems of finite sets, European J. Combin. 18 (1997) 125--136.
		%B
		\bibitem{VHM1} A. Blokhuis, A.E. Brouwer, A. Chowdhury, P. Frankl, T. Mussche, B. Patk\'{o}s and T. Sz\H{o}nyi, A Hilton-Milner theorem for vector spaces, Electron. J. Combin. 17 (2010) \#R71.
		%\bibitem{j1}
		%C. Bey, On cross-intersecting families of sets, Graphs Combin. 21 (2005) 161--168.
		%\bibitem{j2}
		% P. Borg, A short proof of a cross-intersection theorem of Hilton, Discrete Math. 309 (2009) 4750--4753.
		%\bibitem{j3}
		%P. Borg, The maximum sum and the maximum product of sizes of cross intersecting families, European J. Combin. 35 (2014) 117--130.
		
		\bibitem{CROOSBORG2016}
		P. Borg, The maximum product of weights of cross-intersecting families, J. London Math. Soc. 94 (2016) 993--1018.
		\bibitem{BCN} A.E. Brouwer, A.M. Cohen and A. Neumaier, Distance-Regular Graphs, Springer-Verlag, Berlin 1989.
	%	\bibitem{012} A.E. Brouwer, J. Hemmeter, A new family of distance-regular graphs and the $\{0,1,2\}$-cliques in dual polar graphs, European J. Combin. 13(2) (1992) 71--79.
		%C
		\bibitem{VHM2} M. Cao, B. Lv, K. Wang and S. Zhou, Non-trivial $t$-intersecting families for vector spaces, SIAM J. Discrete Math. 36 (2022) 1823--1847.
		\bibitem{CAO} M. Cao, M. Lu, B. Lv and K. Wang, $r$-cross $t$-intersecting families for vector spaces, J. Combin. Theory Ser. A 193 (2023) 105688.
	\bibitem{cao-rcrossset} M. Cao, M. Lu, B. Lv and K. Wang, Nearly extremal non-trivial cross $t$-intersecting families and $r$-wise $t$-intersecting families, European J. Combin. 120 (2024) 103958.
		%D
	%	\bibitem{PLANE} M. De Boeck, The largest Erd\H{o}s-Ko-Rado sets of planes in finite projective and finite classical polar spaces, Des. Codes Cryptogr. 72(1) (2014) 77--117.
	%	\bibitem{HMQ} M. De Boeck, The second largest Erd\H{o}s-Ko-Rado sets of generators of the hyperbolic quadrics $Q^+(4n+1,q)$, Adv. Geom. 16(2) (2016) 253--263.
		%E
		\bibitem{EKR}P. Erd\H{o}s, C. Ko, and R. Rado, Intersection theorems for systems of finite sets, Quart. J. Math. Oxford  Ser. 2 (12) (1961) 313--320.
		%F
		%\bibitem{Fn}P. Frankl, The Erd\H{o}s-Ko-Rado theorem is true for $n=ckt$, in: Combinatorics, Vol. \uppercase\expandafter{\romannumeral1}, Proc. Fifth Hungarian Colloq., Keszthely, 1976, in: Colloq. Math. Soc. J\'{a}nos Bolyai, vol. 18, North-Holland, 1978, pp. 365--375.
		%\bibitem{SHMT1} P. Frankl, On intersecting families of finite sets, J. Combin. Theory, Ser. A 24 (1978) 146--161.
		%\bibitem{FF} P. Frankl and Z. F\"{u}redi, Beyond the Erd\H{o}s-Ko-Rado theorem, J. Combin. Theory Ser. A 56 (1991) 182--194.
		%\bibitem{j4}
		% P. Frankl, S. J. Lee, M. Siggers and N. Tokushige, An Erd\H{o}s-Ko-Rado theorem for cross $t$-intersecting families, J. Combin. Theory Ser. A 128 (2014) 207--249.
		 %\bibitem{j5} P. Frankl and N. Tokushige, Some best possible inequalities concerning cross-intersecting	families, J. Combin. Theory Ser. A 61 (1992) 87--97.
		 %\bibitem{j6} P. Frankl and N. Tokushige, On $r$-cross intersecting families of sets, Combin. Probab. Comput. 20 (2011) 749--752.
		 
		 
		 
		\bibitem{VEKR2} P. Frankl and R. Wilson, The Erd\H{o}s-Ko-Rado theorem for vector spaces, J. Combin. Theory Ser. A 43  (1986) 228--236.
	\bibitem{FW1} P. Frankl and J. Wang, A Product Version of the Hilton-Milner Theorem,  J. Combin. Theory Ser. A 200 (2023) 105791.
	\bibitem{FW2} P. Frankl and J. Wang, A Product Version of the Hilton-Milner-Frankl Theorem, Sci. China Math. 67 (2024) 455--474.
		%G
	%	\bibitem{QY} J. Guo, F. Li and K. Wang, Anzahl formulas of subspaces in symplectic spaces and their applications, Linear Algebra Appl., 438 (2013) 3321--3335.
	%	\bibitem{AEKR}J. Guo and Q. Xu, The Erd\H{o}s-Ko-Rado theorem for finite affine spaces, Linear Multilinear Algebra 65 (2017) 593--599.
		%H
		%\bibitem{LSET3}J. Han and Y. Kohayakawa, The maximum size of a non-trivial intersecting uniform family that is not a subfamily of the Hilton-Milner family, Proc. Amer. Math. Soc. 145(1) (2017) 73--87.
		%\bibitem{j7}
		%A.J.W. Hilton, An intersection theorem for a collection of families of subsets of a finite set, J. London Math. Soc. (2) 15 (1977) 369--376.
		\bibitem{SHM} A. Hilton and E. Milner, Some intersection theorems for systems of finite sets, Quart. J. Math. Oxford Ser. (2) 18 (1967) 369--384.
	%	\bibitem{CPS} J. Hirschfeld and J. Thas, General Galois Geometries, Oxford Mathematical Monographs, Oxford Science.	Publications, The Clarendon Press, Oxford University Press, New York 1991.
		\bibitem{VEKR1} W.N. Hsieh, Intersection theorems for systems of finite vector spaces, Discrete Math. 12 (1975) 1--16.
		%I
		%\bibitem{PHD}  F. Ihringer, Finite geometry intersecting algebraic combinatorics, PhD thesis, University of Giessen, 2015.
		
	%	\bibitem{sp-add}  F. Ihringer, Cross-intersecting Erd\H{o}s-Ko-Rado sets in finite classical polar spaces, Electron. J. Combin. 22 (2) (2015) \#P2.49.
	%	\bibitem{MENTION1} F. Ihringer and K. Metsch, Large $\{0, 1, ..., t\}$-cliques in dual polar graphs, J. Combin. Theory Ser. A 154 (2018) 285--322.
		%K
		%\bibitem{LSET1}A. Kostochka and D. Mubayi, The structure of large intersecting families, Proc. Amer. Math. Soc. 145(6) (2017) 2311--2321.
		%L
	%	\bibitem{SPEKRT} X. Liu, Q. Fan and Q. Sun, Research of the Erd\H{o}s-Ko-Rado theorem based on symplectic spaces over finite fields (in Chinese), J. Hebei Normal Univ.d (Natural Science Edition) 42(4) (2018) 277--283.
		%M
		%\bibitem{j8}
		%M. Matsumoto and N. Tokushige, The exact bound in the Erd\H{o}s-Ko-Rado theorem for cross intersecting families, J. Combin. Theory Ser. A 52 (1989) 90--97.
	%	\bibitem{SPEKR2} K. Metsch, An Erd\H{o}s-Ko-Rado theorem for finite classical polar spaces, J. Algebr. Comb. 43(2)  (2016) 375--397.
		%P
	%	\bibitem{j9}
		%L. Pyber, A new generalization of the Erd\H{o}s-Ko-Rado theorem, J. Combin. Theory Ser. A 43 (1986) 85--90.
	%	\bibitem{SPEKR1} V. Pepe, L. Storme and F. Vanhove, Theorems of Erd\H{o}s-Ko-Rado type in polar spaces, J. Combin. Theory Ser. A 118(4) (2011) 1291--1312.
		%T
		\bibitem{VEKR3} H. Tanaka, Classification of subsets with minimal width and dual width in Grassmann, bilinear forms and dual polar graphs, J. Combin. Theory Ser. A 113 (2006) 903--910.
		
		%\bibitem{j10}
		%N. Tokushige, On cross $t$-intersecting families of sets, J. Combin. Theory Ser. A 117 (2010) 1167--1177.
		
		
	\bibitem{tokushige-set-cross-product}N. Tokushige, The eigenvalue method for cross $t$-intersecting families, J. Algebr. Comb. 38 (2013) 653--662.
		%W
	%	\bibitem{JS1} Z. Wan, Studies in finite geometries and the construction of incomplete block designs.  \uppercase\expandafter{\romannumeral1}. Some `Anzahl' theorems in symplectic geometry over finite fields (in Chinese), Acta Math. Sinica 15 (1965) 354--361.
	%	\bibitem{JS} Z. Wan, Geometry of Classical Groups over Finite Fields, 2nd edition, Science Press, Beijing/New York 2002.
	%	\bibitem{JS2} Z. Wan, Z. Dai, X. Feng and B. Yang, Studies in Finite Geometry and the Construction of Incomplete Block Designs (in Chinese), Science Press, Beijing 1966.
	\bibitem{VHMNEW}
	Y. Wang, A. Xu and J. Yang, A $t$-intersecting Hilton-Milner theorem for vector spaces, Linear Algebra Appl. 680 (2024) 220--238.
	
	\bibitem{VHMNEW2}
	J. Wang, A. Xu and H. Zhang, A Kruskal-Katona-type theorem for graphs: $q$-Kneser graphs, J. Combin. Theory Ser. A 198 (2023) 105766.
	
	%\bibitem{j11}
	%J. Wang and H. Zhang, Cross-intersecting families and primitivity of symmetric systems, J. Combin. Theory Ser. A 118 (2011) 455--462.
	\bibitem{SCROSSSUM2013}
	J. Wang and H. Zhang, Nontrivial independent sets of bipartite graphs and cross-intersecting families, J. Combin. Theory Ser. A 120 (2013) 129--141.
	
		\bibitem{Wn} R. Wilson, The exact bound in the Erd\H{o}s-Ko-Rado theorem, Combinatorica 4 (1984) 247--257.
		%X
\bibitem{kong} Y. Xi, X. Kong and G. Ge, Multi-part cross-intersecting families,	J. Algebr. Comb. 59 (2024) 597--620.	
	%	\bibitem{YT} T. Yao, B. Lv and K. Wang, Non-trivial $t$-intersecting families for symplectic polar spaces, Finite Fields
	%	Appl. 77 (2022) 101955.
	\end{thebibliography}
\end{document}